\documentclass[12pt,reqno]{amsart}
\usepackage{amsmath,amssymb,amsfonts,amscd,latexsym,amsthm,mathrsfs,verbatim}
\usepackage{graphicx}
\usepackage[unicode]{hyperref}
\usepackage{hypbmsec}
\newtheorem{theorem}{Theorem}
\newtheorem{lemma}{Lemma}
\theoremstyle{definition}
\newtheorem{remark}{\bf Remark}

\let\wt\widetilde

\renewcommand{\d}{{\mathrm d}}

\newcommand{\ord}{\operatorname{ord}}

\newcommand{\Beta}{{\boldsymbol\eta}}
\newcommand{\lf}{\lfloor}
\newcommand{\rf}{\rfloor}
\newcommand{\LF}{\lfloor\kern-2.8pt\lfloor}
\newcommand{\RF}{\rfloor\kern-2.8pt\rfloor}
\newcommand{\biggLF}{\biggl\lfloor\kern-3.3pt\biggl\lfloor}
\newcommand{\biggRF}{\biggr\rfloor\kern-3.3pt\biggr\rfloor}
\begin{document}

\title[Arithmetic of Catalan's constant and its relatives]{Arithmetic of Catalan's constant\\and its relatives}

\author{Wadim Zudilin}
\address{Department of Mathematics, IMAPP, Radboud University, PO Box 9010, 6500~GL Nijmegen, Netherlands}
\urladdr{\url{http://www.math.ru.nl/~wzudilin}}

\date{18 April 2018. \emph{Revised}: 4 May 2019}



\dedicatory{To Peter Bundschuh, with many irrational wishes, on the occasion of his 80th birthday}

\begin{abstract}
We prove that at least one of the six numbers $\beta(2i)$ for $i=1,\dots,6$ is irrational.
Here $\beta(s)=\sum_{k=0}^\infty(-1)^k(2k+1)^{-s}$ denotes Dirichlet's beta function, so that $\beta(2)$ is Catalan's constant.
\end{abstract}

\maketitle

\section{Introduction}
\label{s1}

In this note we discuss arithmetic properties of the values of Dirichlet's beta function
$$
\beta(s)=\sum_{n=1}^\infty\frac{\bigl(\frac{-4}n\bigr)}{n^s}=\sum_{k=0}^{\infty}\frac{(-1)^k}{(2k+1)^s}
$$
at positive even integers $s$. The very first such beta value $\beta(2)$ is famously known as Catalan's constant;
its irrationality remains an open problem, though we expect the number to be irrational and transcendental.
The best known results in this direction were given by T.~Rivoal and this author in \cite{RZ03}.
Namely, we showed that at least one of the seven numbers $\beta(2),\beta(4),\dots,\beta(14)$ is irrational,
and that there are infinitely many irrational numbers among the even beta values $\beta(2),\beta(4),\beta(6),\dots$\,.
Here we use a variant of the method from \cite{KZ18,Zu18} to improve slightly on the former achievement;
a significant strengthening towards the infinitude result, based on a further development of the ideas in \cite{FSZ18,Sp18},
is a subject of the recent preprint \cite{Fi18} of S.~Fischler.

\begin{theorem}
\label{th1}
At least one of the six numbers
$$
\beta(2), \; \beta(4), \; \beta(6), \; \beta(8), \; \beta(10), \; \beta(12)
$$
is irrational.
\end{theorem}

In Section~\ref{s2} we illustrate principal ingredients of the method in a particularly simple situation;
this leads to a weaker version of Theorem~\ref{th1}, namely, to the irrationality of at least one number $\beta(2i)$ for $i=1,\dots,8$. The details about the general construction of approximating forms to even beta values and our proof of Theorem~\ref{th1}
are given in Section~\ref{s3}.

\section{Outline of the construction}
\label{s2}

For an odd integer $s\ge3$ (which we eventually set to be 17) and even $n>0$, define the rational function
\begin{equation}
R_n(t)=\frac{2^{6n}n!^{s-3}\,(2t+n)\,\prod_{j=1}^{3n}(t-n+j-\frac12)}{\prod_{j=0}^n(t+j)^s}
\label{rat1}
\end{equation}
and assign to it the related sequence of quantities
\begin{equation}
r_n=\sum_{\nu=1}^\infty R_n(\nu-\tfrac12)(-1)^\nu
=\sum_{\nu=n+1}^\infty R_n(\nu-\tfrac12)(-1)^\nu.
\label{sum1}
\end{equation}
The sums $r_n$ are instances of generalized hypergeometric functions, for which we can use some standard integral representations to write
$$
r_n=\frac{2^{6n}(3n+1)!}{n!^3}\idotsint_{[0,1]^s}
\frac{(1-t_1t_2\dotsb t_s)\,\prod_{j=1}^st_j^{n-1/2}(1-t_j)^n\,\d t_j}
{(1+t_1t_2\dotsb t_s)^{3n+2}}
$$
(details are given in Lemma~\ref{lem2} below). This form clearly implies that $r_n>0$ and also gives access to the asymptotics
\begin{equation}
\lim_{n\to\infty}r_n^{1/n}
=2^63^3\max_{\boldsymbol t\in[0,1]^s}\frac{\prod_{j=1}^st_j(1-t_j)}{(1+t_1t_2\dotsb t_s)^3}
=12^3\max_{0<t<1}\frac{t^s(1-t)^s}{(1+t^s)^3}.
\label{asymp1}
\end{equation}
An important ingredient of the construction is the following decomposition of the quantities $r_n$.

\begin{lemma}
\label{lem1}
For odd $s$ and even $n$ as above,
$$
r_n=\sum_{\substack{i=1\\ i\;\text{even}}}^sa_i\beta(i)+a_0,
$$
where $a_i=a_{i,n}$ satisfy the inclusions $\Phi_n^{-1}d_n^{s-i}a_i\in\mathbb Z$ for $i=0,1,\dots,s$ even.
Here $d_n$ denotes the least common multiple of the numbers $1,2,\dots,n$, and
$$
\Phi_n=\prod_{2\sqrt n<p\le n}p^{\varphi_0(n/p)}
\qquad\text{with}\quad
\varphi_0(x)=\begin{cases}
0 & \text{if}\;\; 0\le\{x\}<\tfrac13, \\[2pt]
1 & \text{if}\;\; \tfrac13\le\{x\}<\tfrac12, \\[2pt]
2 & \text{if}\;\; \tfrac12\le\{x\}<1,
\end{cases}
$$
the product taken over primes.
\end{lemma}

Note that from the prime number theorem we deduce the asymptotics
$$
\lim_{n\to\infty}d_n^{1/n}=e
\qquad\text{and}\qquad
\lim_{n\to\infty}\Phi_n^{1/n}=e^{\varkappa},
$$
where
$$
\varkappa=\bigl(\psi(\tfrac12)-\psi(\tfrac13)-1\bigr)+2\cdot\bigl(\psi(1)-\psi(\tfrac12)-1\bigr)
=0.9411124762\dots,
$$
the function $\psi(x)$ denotes the logarithmic derivative of the gamma function.

\begin{remark}
\label{rem1}
The analogous construction in \cite{RZ03} makes use of a slightly different rational function than \eqref{rat1}, namely, of
$$
\wt R_n(t)=\frac{2^{4n}n!^{s-2}\,(2t+n)\,\prod_{j=1}^n(t-n+j-\frac12)\,\prod_{j=1}^n(t+n+j-\frac12)}{\prod_{j=0}^n(t+j)^s},
$$
so that
$$
R_n(t)=\wt R_n(t)\cdot\frac{2^{2n}\prod_{j=1}^n(t+j-\frac12)}{n!}.
$$
The analogous decomposition of a related quantity $\wt r_n$ assumes the form
$$
\wt r_n=\sum_{\substack{i=1\\ i\;\text{even}}}^s\wt a_i\beta(i)+\wt a_0,
$$
in which the rational coefficients $\wt a_i=\wt a_{i,n}$ satisfy $\Phi_n^{-1}d_n^{s-i}\wt a_i\in\mathbb Z$ for $i=1,\dots,s$ even, but $\Phi_n^{-1}d_{2n}^s\wt a_0\in\mathbb Z$. The appearance of $d_{2n}^s$ as the common denominator in place of $d_n^s$ changes the scene drastically and leads to weaker arithmetic applications.
\end{remark}

\begin{proof}[Proof of Lemma~\textup{\ref{lem1}}]
Following the strategy in \cite{RZ18,Zu18} we can write the function \eqref{rat1} as sum of partial fractions,
$$
R_n(t)=\sum_{i=1}^s\sum_{k=0}^n\frac{a_{i,k}}{(t+k)^i},
$$
in which $\Phi_n^{-1}d_n^{s-i}a_{i,k}\in\mathbb Z$ for all $i$ and $k$.
Indeed, the rational function \eqref{rat1} is a product of simpler ones
\begin{align*}
\frac{n!}{\prod_{j=0}^n(t+j)}
&=\sum_{k=0}^n\frac{(-1)^k\binom nk}{t+k},
\displaybreak[2]\\
\frac{2^{2n}\prod_{j=1}^n(t-n+j-\frac12)}{\prod_{j=0}^n(t+j)}
&=\sum_{k=0}^n\frac{(-1)^{n+k}\binom{2n+2k}{2n}\binom{2n}{n+k}}{t+k},
\displaybreak[2]\\
\frac{2^{2n}\prod_{j=1}^n(t+j-\frac12)}{\prod_{j=0}^n(t+j)}
&=\sum_{k=0}^n\frac{\binom{2k}k\binom{2n-2k}{n-k}}{t+k},
\displaybreak[2]\\
\frac{2^{2n}\prod_{j=1}^n(t+n+j-\frac12)}{\prod_{j=0}^n(t+j)}
&=\sum_{k=0}^n\frac{(-1)^k\binom{4n-2k}{2n}\binom{2n}k}{t+k}
\end{align*}
and $2t+n$; the inclusions $d_n^{s-i}a_{i,k}\in\mathbb Z$ then follow from \cite[Lemma~1]{Zu18}.
The cancellation by the factor $\Phi_n$ originates from the $p$-adic analysis of the binomial factors entering
$$
a_{s,k}=\frac{(2n+2k)!\,(4n-2k)!}{(n+k)!\,(2n-k)!\,k!^3(n-k)!^3}\cdot(n-2k){\binom nk}^{s-3}
\quad\text{for} \; k=0,1,\dots,n,
$$
and the estimate $\ord_pa_{i,k}\ge-(s-i)+\ord_pa_{s,k}\ge-(s-i)+\varphi(n/p,k/p)$ for primes in the range $2\sqrt n<p\le n$,
where
$$
\varphi(x,y)=\lf 2x+2y\rf+\lf 4x-2y\rf-\lf x+y\rf-\lf 2x-y\rf-3\lf y\rf-3\lf x-y\rf
$$
is a periodic function of period 1 in both $x$ and $y$, and from the inequality
$$
\varphi(x,y)\ge\min_{y\in\mathbb R}\varphi(x,y)=\min_{0\le y<1}\varphi(x,y)=\varphi_0(x);
$$
the details can be borrowed from~\cite[Section~7]{RZ03}.
Furthermore, the property $R_n(-t-\nobreak n)=R_n(t)$ derived from \eqref{rat1} implies $a_{i,k}=(-1)^ia_{i,n-k}$ for $i=1,\dots,s$ and $k=0,1,\dots,n$.

Recall that $n$ is even, so that $n/2=m$ is a positive integer.
The summation over $\nu$ in \eqref{sum1} can also start from $-m-1$ (rather than $1$ or $n+1$), because
the function $R_n(t)$ vanishes at all half-integers between $-2n$ and $n$. Therefore,
\begin{align*}
r_n
&=\sum_{\nu=-m-1}^\infty R_n(\nu-\tfrac12)(-1)^\nu
=\sum_{i=1}^s\sum_{k=0}^na_{i,k}\sum_{\nu=-m-1}^\infty\frac{(-1)^\nu}{(\nu+k-\frac12)^i}
\\
&=\sum_{i=1}^s\sum_{k=0}^n(-1)^{k-1}a_{i,k}
\sum_{\ell=k-m}^\infty\frac{(-1)^\ell}{(\ell+\frac12)^i}
\displaybreak[2]\\
&=\sum_{i=1}^s2^i\beta(i)\sum_{k=0}^n(-1)^{k-1}a_{i,k}
+\sum_{i=1}^s\sum_{k=0}^{m-1}(-1)^{k-1}a_{i,k}\sum_{\ell=k-m}^{-1}\frac{(-1)^\ell}{(\ell+\frac12)^i}
\\ &\qquad
-\sum_{i=1}^s\sum_{k=m+1}^n(-1)^{k-1}a_{i,k}\sum_{\ell=0}^{k-m-1}\frac{(-1)^\ell}{(\ell+\frac12)^i},
\end{align*}
where the rules
$$
\sum_{\ell=k-m}^\infty=\sum_{\ell=0}^\infty-\sum_{\ell=0}^{k-m-1} \quad\text{if}\; k>m
\qquad\text{and}\qquad
\sum_{\ell=k-m}^\infty=\sum_{\ell=0}^\infty+\sum_{\ell=k-m}^{-1} \quad\text{if}\; k<m
$$
were applied. Thus, the rational numbers
$$
a_i=2^i\sum_{k=0}^n(-1)^{k-1}a_{i,k} \quad\text{for}\; i=1,\dots,s
$$
satisfy $\Phi_n^{-1}d_n^{s-i}a_i\in\mathbb Z$, while for the quantity
$$
a_0=\sum_{i=1}^s\sum_{k=0}^{m-1}(-1)^{k-1}a_{i,k}\sum_{\ell=k-m}^{-1}\frac{(-1)^\ell}{(\ell+\frac12)^i}
-\sum_{i=1}^s\sum_{k=m+1}^n(-1)^{k-1}a_{i,k}\sum_{\ell=0}^{k-m-1}\frac{(-1)^\ell}{(\ell+\frac12)^i}
$$
the inclusion $\Phi_n^{-1}d_n^sa_0\in\mathbb Z$ follows from noticing that
\begin{align*}
d_{n-1}^i\sum_{\ell=k-m}^{-1}\frac{(-1)^\ell}{(\ell+\frac12)^i}&\in\mathbb Z \quad\text{if}\; 0\le k<m,
\\
d_{n-1}^i\sum_{\ell=0}^{k-m-1}\frac{(-1)^\ell}{(\ell+\frac12)^i}&\in\mathbb Z \quad\text{if}\; m<k\le n.
\end{align*}

Finally,
\begin{align*}
a_i=2^i\sum_{k=0}^n(-1)^{k-1}a_{i,k}
&=(-1)^i2^i\sum_{k=0}^n(-1)^{k-1}a_{i,n-k}
\\
&=(-1)^i2^i\sum_{k=0}^n(-1)^{n-k-1}a_{i,n-k}
=(-1)^ia_i,
\end{align*}
so that $a_i$ vanish for odd $i$.
\end{proof}

Set now $s=17$, in which case we compute from \eqref{asymp1} that
$$
\lim_{n\to\infty}r_n^{1/n}=e^{-16.1123070755\dots},
$$
hence the linear forms
$$
\Phi_n^{-1}d_n^{17}r_n\in\mathbb Z\beta(2)+\mathbb Z\beta(4)+\dots+\mathbb Z\beta(16)+\mathbb Z
$$
are positive and tend to 0 as $n\to\infty$. This implies that the eight numbers $\beta(2),\beta(4),\dots,\beta(16)$ cannot
be all rational.

\section{General settings}
\label{s3}

A natural way to generalize the construction in Section~\ref{s2} follows the recipe of \cite{RZ03} and \cite{Zu04}.

For an odd integer $s\ge5$, consider a collection $\Beta=(\eta_0,\eta_1,\dots,\eta_s)$ of integral parameters satisfying the conditions
\begin{equation*}
0<\eta_j<\frac12\eta_0
\quad\text{for}\; j=1,\dots,s
\qquad\text{and}\qquad
\eta_1+\eta_2+\dots+\eta_s\le\frac{s-1}{2}\eta_0,
\end{equation*}
to which we assign, for each positive integer $n$, the collection
\begin{equation*}
h_0=\eta_0n+1,\quad
h_j=\eta_jn+\tfrac12 \quad\text{for}\; j=1,\dots,s.
\end{equation*}
In what follows, we assume that $h_0-1=\eta_0n$ is even\,---\,the condition that is automatically achieved when $\eta_0\in2\mathbb Z$, otherwise by restricting to even $n$.

Define the rational function
\begin{equation}
R_n(t)=R_{n,\Beta}(t)
=\gamma_n\cdot(2t+h_0)\,
\frac{(t+1)_{h_0-1}}{\prod_{j=1}^s(t+h_j)_{1+h_0-2h_j}}=R_n(-t-h_0),
\label{rat2}
\end{equation}
where
$$
\gamma_n=4^{h_0-1}\,\frac{\prod_{j=2}^s(h_0-2h_j)!}{(h_1-\frac12)!^2},
$$
and the (very-well-poised) hypergeometric sum
\begin{align}
r_n
&=r_{n,\Beta}
=\sum_{\nu=0}^\infty R_n(\nu)(-1)^\nu
\nonumber\\
&=\gamma_n
\cdot\frac{\Gamma(1+h_0)\,\prod_{j=1}^q\Gamma(h_j)}
{\prod_{j=1}^q\Gamma(1+h_0-h_j)}
\,{}_{s+2}F_{s+1}
\biggl(\begin{matrix}
h_0, \; 1+\tfrac12h_0, \; h_1, \; \dots,\; h_s \\[1pt]
\tfrac12h_0, \, 1+h_0-h_1, \, \dots, \, 1+h_0-h_s \end{matrix}\biggm|-1\biggr).
\label{eq:5}
\end{align}
Then \cite[Lemma 1]{RZ03} implies the following Euler-type integral representation of~$r_n$
(see also \cite[Lemma 3]{RZ03}).

\begin{lemma}
\label{lem2}
The formula
\begin{align*}
r_n
&=\frac{4^{h_0-1}\Gamma(1+h_0)}{\Gamma(h_1+\frac12)^2\,\Gamma(1+h_0-2h_1)}
\\ &\qquad\times
\idotsint_{[0,1]^s}
\frac{\prod_{j=1}^st_j^{h_j-1}(1-t_j)^{h_0-2h_j}}
{(1+t_1t_2\dotsb t_s)^{h_0}}\,
\frac{1-t_1t_2\dotsb t_s}{1+t_1t_2\dotsb t_s}
\,\d t_1\,\d t_2\dotsb\d t_s
\end{align*}
is valid. In particular, $r_n>0$ and
$$
\lim_{n\to\infty}r_n^{1/n}=\frac{(4\eta_0)^{\eta_0}}{\eta_1^{2\eta_1}(\eta_0-2\eta_1)^{\eta_0-2\eta_1}}
\cdot\max_{\boldsymbol t\in[0,1]^s}\frac{\prod_{j=1}^st_j^{\eta_j}(1-t_j)^{\eta_0-2\eta_j}}{(1+t_1t_2\dotsb t_s)^{\eta_0}}.
$$
\end{lemma}

Computation of the latter maximum is performed in \cite[Sect.~4, Remark]{RZ03}, and the result is as follows.

\begin{lemma}
\label{lem3}
Assume that $x_0$ is a \emph{unique} zero of the polynomial
\begin{equation*}
x\prod_{j=1}^s\bigl((\eta_0-\eta_j)-\eta_jx\bigr)-\prod_{j=1}^s\bigl(\eta_j-(\eta_0-\eta_j)x\bigr)
\end{equation*}
in the interval $0<x<1$, and set
$$
x_j
=\frac{\eta_j-(\eta_0-\eta_j)x_0}{(\eta_0-\eta_j)-\eta_jx_0}
\quad\text{for}\; j=1,2,\dots,s.
$$
Then
$$
\max_{\boldsymbol t\in[0,1]^s}\frac{\prod_{j=1}^st_j^{\eta_j}(1-t_j)^{\eta_0-2\eta_j}}{(1+t_1t_2\dotsb t_s)^{\eta_0}}
=\frac{\prod_{j=1}^sx_j^{\eta_j}(1-x_j)^{\eta_0-2\eta_j}}{(1+x_1x_2\dotsb x_s)^{\eta_0}}.
$$
\end{lemma}

Arithmetic ingredients of the construction are in line with the strategy used in the proof of Lemma~\ref{lem1}.
For simplicity we split the corresponding statement into two parts. Define
$$
N=\min_{1\le j\le s}\{h_j-\tfrac12\}
\qquad\text{and}\qquad
M=\max\{h_0-2N-1,\,h_1-\tfrac12\},
$$
and notice that the poles of the rational function \eqref{rat2} are located at the points $t=-k-\frac12$ for integers
$k$ in the range $N\le k\le h_0-N-1$.

\begin{lemma}
\label{lem4}
The coefficients in the partial-fraction decomposition
$$
R_n(t)=\sum_{i=1}^s\sum_{k=N}^{h_0-N-1}\frac{a_{i,k}}{(t+k+\tfrac12)^i}
$$
of \eqref{rat2} satisfy
\begin{equation}
a_{i,k}=(-1)^ia_{i,h_0-1-k}
\label{vanish2}
\end{equation}
and
\begin{equation}
\Phi_n^{-1}d_M^{s-i}a_{i,k}\in\mathbb Z
\label{arith2}
\end{equation}
for $i=1,\dots,s$ and $N\le k\le h_0-N-1$, where the product over primes
\begin{equation*}
\Phi_n=\prod_{\sqrt{2h_0}<p\le M}p^{\varphi_0(n/p)}
\end{equation*}
is defined through the $1$-periodic functions
$$
\varphi_0(x)=\min_{0\le y<1}\varphi(x,y)
$$
and
\begin{align*}
\varphi(x,y)
&=\lf2(\eta_0x-y)\rf+\lf2y\rf-\lf\eta_0x-y\rf-\lf y\rf-2\lf\eta_1x\rf-\lf(\eta_0-2\eta_1)x\rf
\\ &\quad
+\sum_{j=1}^s\bigl(\lf(\eta_0-2\eta_j)x\rf-\lf y-\eta_jx\rf-\lf(\eta_0-\eta_j)x-y\rf\bigr).
\end{align*}
\end{lemma}

\begin{proof}
For this, we write the function $R_n(t-\tfrac12)$ as the product of $2t+h_0-1$, the three integer-valued polynomials
$$
\frac{4^{h_1^*}(t+\frac12)_{h_1^*}}{h_1^*!}, \quad
\frac{4^{h_0-2h_1}(t+h_1^*+\frac12)_{h_0-2h_1}}{(h_0-2h_1)!}, \quad
\frac{4^{h_1^*}(t+h_0-h_1^*-\frac12)_{h_1^*}}{h_1^*!},
$$
where $h_1^*=h_1-\frac12=\eta_1n$, and the rational functions
$$
\frac{(h_0-2h_j)!}{(t+h_j-\frac12)_{1+h_0-2h_j}}
\quad\text{for}\; j=1,\dots,s.
$$
Then \cite[Lemmas 4, 5, 10, 11]{RZ03} and the Leibniz rule for differentiating a product
imply the inclusions $d_M^{s-i}a_{i,k}\in\mathbb Z$ and estimates
$$
\ord_pa_{i,k}\ge-(s-i)+\varphi\Bigl(\frac np,\frac kp\Bigr)
$$
for the $p$-adic order of the coefficients. These are combined to conclude with \eqref{arith2}.

The property \eqref{vanish2} follows from the symmetry of the rational function \eqref{rat2}.
\end{proof}

\begin{lemma}
\label{lem5}
The decomposition
\begin{equation}
r_n=\sum_{\substack{i=1\\ i\;\text{even}}}^sa_i\beta(i)+a_0
\in\mathbb Q\beta(2)+\mathbb Q\beta(4)+\dots+\mathbb Q\beta(s-1)+\mathbb Q
\label{dec2}
\end{equation}
takes place, where $\Phi_n^{-1}d_M^{s-i}a_i\in\mathbb Z$ for $i=0,1,\dots,s$ even,
and $\Phi_n$ is defined in Lemma~\textup{\ref{lem4}}.
\end{lemma}

\begin{proof}
Since the function \eqref{rat2} vanishes at $t=-1,-2,\dots,-h_0+2$, we can shift the summation in \eqref{eq:5}:
$$
r_n
=\sum_{\nu=-(h_0-1)/2}^\infty R_n(\nu)(-1)^\nu
=\sum_{i=1}^s\sum_{k=N}^{h_0-N-1}(-1)^ka_{i,k}\sum_{\nu=-(h_0-1)/2}^\infty\frac{(-1)^{\nu+k}}{(\nu+k+\tfrac12)^i}.
$$
Now proceeding as in the proof of Lemma~\ref{lem1} we arrive at the desired decomposition \eqref{dec2} with
\begin{align*}
a_0
&=\sum_{i=1}^s\sum_{k=N}^{(h_0-3)/2}(-1)^ka_{i,k}\sum_{\ell=k-(h_0-1)/2}^{-1}\frac{(-1)^\ell}{(\ell+\tfrac12)^i}
\\ &\qquad
-\sum_{i=1}^s\sum_{k=(h_0+1)/2}^{h_0-N-1}(-1)^ka_{i,k}\sum_{\ell=0}^{k-(h_0+1)/2}\frac{(-1)^\ell}{(\ell+\tfrac12)^i}
\end{align*}
and
$$
a_i=2^i\sum_{k=N}^{h_0-N-1}(-1)^ka_{i,k},
$$
with $a_i$ vanishing for $i$ odd in view of the property \eqref{vanish2}. The inclusions for the coefficients in \eqref{dec2}
therefore follow from Lemma~\ref{lem4} and
\begin{align*}
d_{h_0-2N-2}^i\sum_{\ell=k-(h_0-1)/2}^{-1}\frac{(-1)^\ell}{(\ell+\tfrac12)^i}&\in\mathbb Z
\quad\text{if}\;\; N\le k\le\frac{h_0-3}2,
\\
d_{h_0-2N-2}^i\sum_{\ell=0}^{k-(h_0+1)/2}\frac{(-1)^\ell}{(\ell+\tfrac12)^i}&\in\mathbb Z
\quad\text{if}\;\; \frac{h_0+1}2\le k\le h_0-N-1.
\qedhere
\end{align*}
\end{proof}

\begin{proof}[Proof of Theorem~\textup{\ref{th1}}]
Take $s=13$ and
$$
(\eta_0,\eta_1,\dots,\eta_{13})=(31,10,10,10,10,10,11,11,11,11,12,12,12,12),
$$
hence $M=11n$. Then
$$
\lim_{n\to\infty}r_n^{1/n}=\exp(-100.73966317\dotsb)
$$
and
$$
\varphi_0(x)=\begin{cases}
8  & \text{if}\; \{x\}\in\bigl[\frac7{24},\frac3{10}\bigr), \\
7  & \text{if}\; \{x\}\in\bigl[\frac3{31},\frac1{10}\bigr)
\cup\bigl[\frac6{31},\frac15\bigr)
\cup\bigl[\frac9{31},\frac7{24}\bigr)
\cup\bigl[\frac{19}{24},\frac45\bigr)
\cup\bigl[\frac89,\frac9{10}\bigr), \\
6  & \text{if}\; \{x\}\in\bigl[\frac1{11},\frac3{31}\bigr)
\cup\bigl[\frac2{11},\frac6{31}\bigr)
\cup\bigl[\frac3{11},\frac9{31}\bigr)
\cup\bigl[\frac7{20},\frac25\bigr)
\cup\bigl[\frac9{20},\frac12\bigr)
\\ &\qquad\qquad
\cup\bigl[\frac{11}{20},\frac35\bigr)
\cup\bigl[\frac{13}{20},\frac7{10}\bigr)
\cup\bigl[\frac34,\frac{19}{24}\bigr)
\cup\bigl[\frac67,\frac89\bigr), \\
5  & \text{if}\; \{x\}\in\bigl[\frac7{31},\frac14\bigr)
\cup\bigl[\frac7{22},\frac7{20}\bigr)
\cup\bigl[\frac37,\frac9{20}\bigr)
\cup\bigl[\frac{17}{31},\frac{11}{20}\bigr)
\cup\bigl[\frac{19}{31},\frac58\bigr)
\\ &\qquad\qquad
\cup\bigl[\frac{20}{31},\frac{13}{20}\bigr)
\cup\bigl[\frac57,\frac8{11}\bigr)
\cup\bigl[\frac{23}{31},\frac34\bigr), \\
4  & \text{if}\; \{x\}\in\bigl[\frac1{10},\frac18\bigr)
\cup\bigl[\frac15,\frac7{31}\bigr)
\cup\bigl[\frac3{10},\frac7{22}\bigr)
\cup\bigl[\frac6{11},\frac{17}{31}\bigr)
\cup\bigl[\frac35,\frac{19}{31}\bigr)
\\ &\qquad\qquad
\cup\bigl[\frac7{11},\frac{20}{31}\bigr)
\cup\bigl[\frac8{11},\frac{23}{31}\bigr), \\
2  & \text{if}\; \{x\}\in\bigl[\frac1{24},\frac1{11}\bigr)
\cup\bigl[\frac3{22},\frac2{11}\bigr)
\cup\bigl[\frac14,\frac3{11}\bigr)
\cup\bigl[\frac{13}{24},\frac6{11}\bigr)
\cup\bigl[\frac58,\frac7{11}\bigr)
\\ &\qquad\qquad
\cup\bigl[\frac{17}{20},\frac67\bigr)
\cup\bigl[\frac{19}{20},1\bigr), \\
1  & \text{if}\; \{x\}\in\bigl[\frac1{31},\frac1{24}\bigr)
\cup\bigl[\frac18,\frac3{22}\bigr)
\cup\bigl[\frac9{22},\frac37\bigr)
\cup\bigl[\frac12,\frac{13}{24}\bigr)
\cup\bigl[\frac7{10},\frac57\bigr)
\\ &\qquad\qquad
\cup\bigl[\frac45,\frac9{11}\bigr)
\cup\bigl[\frac{26}{31},\frac{17}{20}\bigr)
\cup\bigl[\frac9{10},\frac{10}{11}\bigr)
\cup\bigl[\frac{29}{31},\frac{19}{20}\bigr), \\
0 & \text{otherwise},
\end{cases}
$$
so that
$$
\lim_{n\to\infty}(\Phi_n^{-1}d_M^{13})^{1/n}=\exp(100.23354349\dots).
$$
This means that the positive linear forms
$$
\Phi_n^{-1}d_M^{13}r_n\in\mathbb Z\beta(2)+\mathbb Z\beta(4)+\dots+\mathbb Z\beta(12)+\mathbb Z
$$
tend to 0 as $n\to\infty$. Thus, at least one of the even beta values in consideration must be irrational.
\end{proof}


\end{document}